\documentclass[a4paper,12pt]{amsart}

\usepackage[margin=3cm]{geometry}
\usepackage[pdftex]{graphicx}
\usepackage[utf8]{inputenc}
\usepackage{amsthm}
\usepackage{amsmath}
\usepackage{amssymb}
\usepackage{color}
\usepackage{amsaddr}
\usepackage{hyperref}
\usepackage{pgf,tikz}
\usetikzlibrary{arrows}

\theoremstyle{plain}
\newtheorem{theorem}{Theorem}[section]
\newtheorem{corollary}[theorem]{Corollary}
\newtheorem{lemma}[theorem]{Lemma}
\newtheorem{proposition}[theorem]{Proposition}
\newtheorem{definition}{Definition} [section]
\newtheorem{problem}{Problem} [section]
\newtheorem{example}{Example} [section]
\newtheorem{remark}{Remark} [section]

\title{Point-Ellipse and some other Exotic Configurations}

\author{G{\'a}bor G{\'e}vay}
\address{Bolyai Institute, University of Szeged, Hungary}
\email{gevay@math.u-szeged.hu}

\author{Nino Ba{\v s}i{\'c}, \quad Jurij Kovi\v{c}}
\address{FAMNIT, University of Primorska, Slovenia \\
IAM, University of Primorska, Slovenia \\
IMFM, Slovenia}
\email{nino.basic@famnit.upr.si}
\email{jurij.kovic@siol.net}

\author{Toma\v{z} Pisanski}
\address{IAM, University of Primorska, Slovenia \\
FAMNIT, University of Primorska, Slovenia \\
IMFM, Slovenia \\
FMF, University of Ljubljana, Slovenia}
\email{tomaz.pisanski@upr.si}

\thanks{The authors acknowledge support by the ARRS grant N1-0032, Slovenia and by the OTKA grant NN-114614, Hungary.
The second, third, and fourth author were also supported in part by the ARRS, grants P1-0294, N1-0011, J1-6720, and J1-7051.}  

\begin{document}

\nocite{*}
\maketitle

\begin{abstract}
In this paper we introduce \emph{point-ellipse configurations} and \emph{point-conic configurations}.
We study some of their basic properties and describe two interesting families of balanced point-ellipse, 
respectively point-conic 6-configurations. The construction of the first family is based on Carnot's theorem, 
whilst the construction of the second family is based on the Cartesian product of two regular polygons. 
Finally, we investigate a point-ellipse configuration based on the regular 24-cell.
\end{abstract}

\vspace{\baselineskip}

\medskip

{\small \noindent
\emph{Math.\ Subj.\ Class.\ (2010):} 51A20 (Configuration theorems), 05B30 (Other designs, configurations)
}

\section{Introduction}

For an introductory text to configurations see the monograph by Gr\"unbaum~\cite{grunbaum2009} which
deals primarily with \emph{geometric point-line configurations}, or the book by Pisanski and Servatius~\cite{pisanski2013}
which takes a more combinatorial approach.
\begin{definition}
A \emph{geometric point-line $(p_q, n_k)$ configuration} is a family of $p$ points and $n$ lines, usually in real projective plane, 
such that each point is incident with precisely $q$ lines and each line is incident with precisely $k$ points.
\end{definition}
 If $p = n$ (and therefore $q = k$), then
the configuration is called \emph{balanced}. Instead of lines, one can also use some other geometric objects. For example, if circles
are used instead of lines, then \emph{point-circle configurations}~\cite{gevay2009, gevay2014} are obtained. Here we investigate 
\emph{point-ellipse configurations} and their generalisation, the so-called \emph{point-conic configurations}.

Configurations can also be studied in a purely combinatorial manner. The traditional definition is:
\begin{definition}
A \emph{combinatorial $(p_q, n_k)$-configuration} is an incidence structure 
$\mathcal{C} = ({P},{B},{I})$, where 
$I \subseteq {P} \times B$, $P \cap B = \emptyset$, $|P| = p$ and $|B| = n$. 
The elements of $P$ are called \emph{points}, the elements of ${B}$ are called \emph{blocks}, 
and the relation ${I}$ is called the \emph{incidence} relation. Furthermore, each line is incident 
with $k$ points, each point is incident with $q$ blocks, and two distinct points are incident with 
at most one common line, i.e.,
\begin{equation}
\{ (p_1, b_1), (p_2, b_1), (p_1, b_2), (p_2, b_2) \} \subseteq I, p_1 \neq p_2 \implies b_1 = b_2.
\label{eq:girthCondition}
\end{equation}
\end{definition}

The bipartite graph on the vertex set $P \cup B$ having an edge between $p \in P$ and $b \in B$ if and only if the elements
$p$ and $b$ are incident in $\mathcal{C}$, i.e., if $(p,b) \in I$, is called the \emph{Levi graph} of configuration $\mathcal{C}$ 
and is denoted by $L(\mathcal{C})$. Moreover, any configuration is completely determined by a $(q,k)$-regular bipartite graph
with a given black-and-white vertex colouring, where black vertices correspond to points and white vertices correspond to blocks.
Such a graph will be called a \emph{coloured Levi graph}. Note that the reverse colouring determines the dual configuration 
$\mathcal{C}^* = (B, P, I^*)$. 
Also, an isomorphism between configurations corresponds to colour-preserving isomorphism between their respective coloured 
Levi graphs.

Recall that a Levi graph of a point-line configuration has to have girth at least~6.

\begin{definition}
A combinatorial configuration whose Levi graph has girth at least 6 is called \emph{lineal}.
\end{definition}

The term lineal has been used by Branko Gr\"unbaum ~\cite {grunbaum} and also in~\cite{pisanski2013}. 
It is not hard to see that lineal Levi graphs can be characterized by forbidden subgraphs.

\begin{proposition}
A combinatorial configuration $\mathcal{C}$ is lineal if and only if $L(\mathcal{C})$ contains no $K_{2,2}$ subgraph.
\end{proposition}

Note that lineality is preserved by duality

\begin{proposition}
A combinatorial configuration $\mathcal{C}$ is lineal if and only if its dual is lineal.
\end{proposition}

When carrying concepts over from point-line configurations to point-circle configurations, one has to take into account 
that two circles may intersect in at most two points, since when intersecting in three points they must coincide. 

\begin{definition}
A combinatorial configuration  $\mathcal{C}$ such that any two combinatorial 
circles meet in at most two combinatorial points is called \emph{circular}.
\end{definition}

Analogously to the case of combinatorial lineal configurations that can be characterized by forbidden subgraphs, 
combinatorial circular configurations admit a similar forbidden subgraph characterization.

\begin{proposition} \label{prop:circular}
A combinatorial configuration $\mathcal{C}$ is circular if and only if the coloured Levi graph $L(\mathcal{C})$ contains no $K_{3,2}$ subgraph.
\end{proposition}

Note that here, and later, any reference to $K_{m,n}$ refers to $m$ points alias black vertices, and $n$ blocks alias white vertices.
However, here the similarity between lineal and circular configurations comes to an end. In particular, circularity is 
not preserved by duality. Namely, the Levi graph of a circular configuration does not contain $K_{3,2}$, however it 
may contain $K_{2,3}$, or even $K_{2,k}, k >3$. Although the notion of circularity need not be preserved by duality, 
it is interesting to consider circular configurations having a circular dual.

\begin{definition}
A combinatorial circular configuration  $\mathcal{C}$ having circular dual is called \emph{strongly circular}.
\end{definition}

From Proposition~\ref{prop:circular} the following characterization of strongly circular configurations follows readily.

\begin{proposition}
A combinatorial configuration $\mathcal{C}$ is strongly circular if and only if the coloured Levi graph $L(\mathcal{C})$ contains 
neither $K_{3,2}$ nor $K_{2,3}$ subgraph. This, in turn, is equivalent to saying that any two 4-cycles in 
$L(\mathcal{C})$ may share at most one edge.
\end{proposition}

In the past the emphasis in the study of configurations has been on point-line configurations. 
Nevertheless there are several works dealing with point-circle configurations; see for instance 
\cite{boben2015,gevay2014,gevay2018b,gevay2018a,gevay2019}.  For unusual isometric point-circle configurations,
see \cite{izquierdo2016,stokes2016}.

It is well-known that there exist many lineal configurations that are not linear. In other words, there are
combinatorial point-line configurations, such as the Fano plane, that do not admit a geometric realization. In general, 
the problem which lineal configurations are linear remains open. One may pose an analogous problem for circular 
configurations. 
\begin{problem}
Which circular combinatorial configurations admit geometric realization by points and circles?
\end{problem}
\begin{figure}[!h]
\centering
    \includegraphics[width=0.85\textwidth]{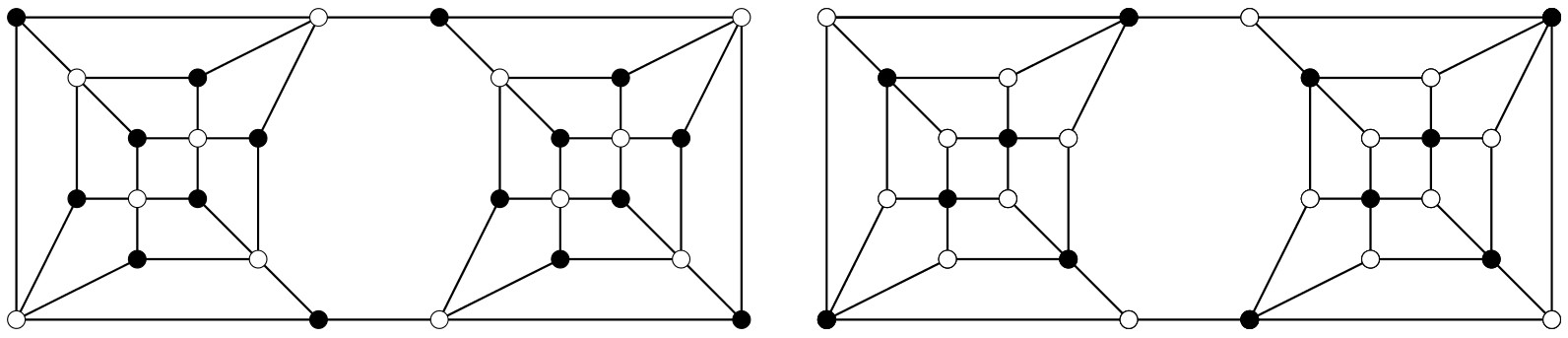}
\caption{On the left, the Levi graph of a circular $(16_3, 12_4)$ configuration that cannot be realized as geometric configuration of points and circles. 
On the right the same graph with colors switched is the Levi graph of the dual $(12_4,16_3)$ configuration that can be realized by points and circles.}
\end{figure}

\begin{proposition}
There exist combinatorial circular configurations that have no geometric realization as point-circle configurations.
\end{proposition}

\begin{proof}
For instance, take two copies of the Miquel configuration and make an appropriate \emph{incidence switch} between 
them. For the definition of the incidence switch, see \cite{pisanski2013}. The resulting Levi graph of the corresponding 
point-circle $(16_3,12_4)$ configuration is depicted in Figure~\ref{fig:AntiMiquel}. It is called the \emph{small Anti-Miquel configuration}.
\end{proof}

\begin{figure}[!h]
\centering
    \includegraphics[width=0.45\textwidth]{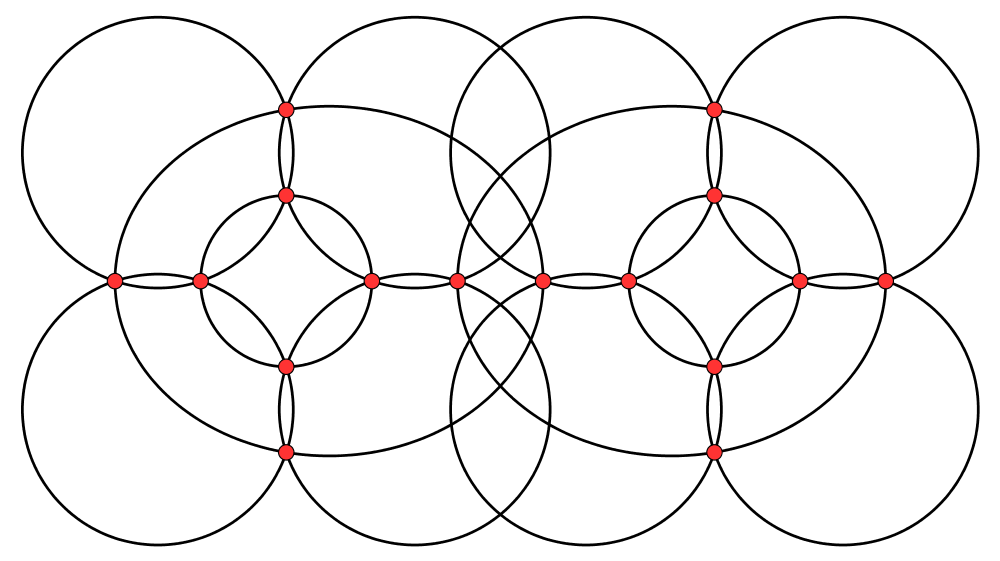}
    \includegraphics[width=0.465\textwidth]{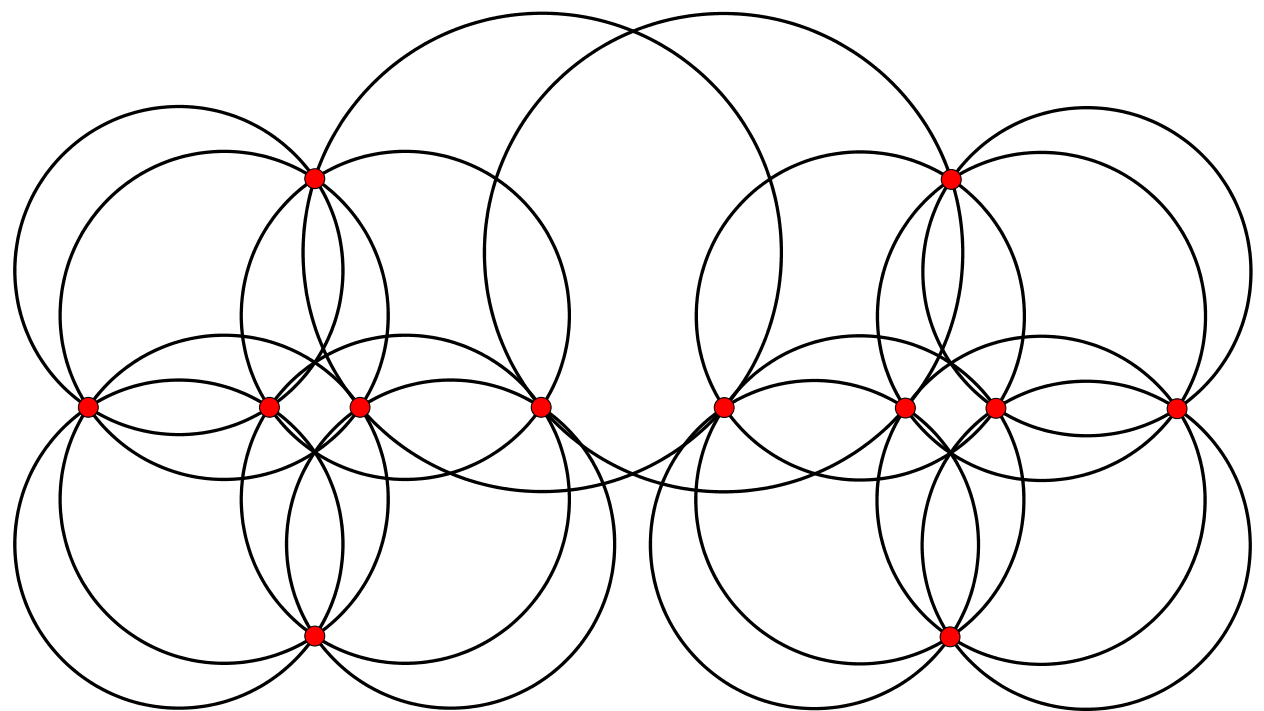}
\caption{Left: A point-ellipse realization of the circular $(16_3,12_4)$ configuration derived from the previous Figure. 
           Note that only two ellipses are used. The rest are circles. We call it the small Anti-Miquel configuration. Right: Its dual admits a point-circle realization.}
\label{fig:AntiMiquel}
\end{figure}

\begin{figure}[!h]
\centering
    \includegraphics[width=0.5\textwidth]{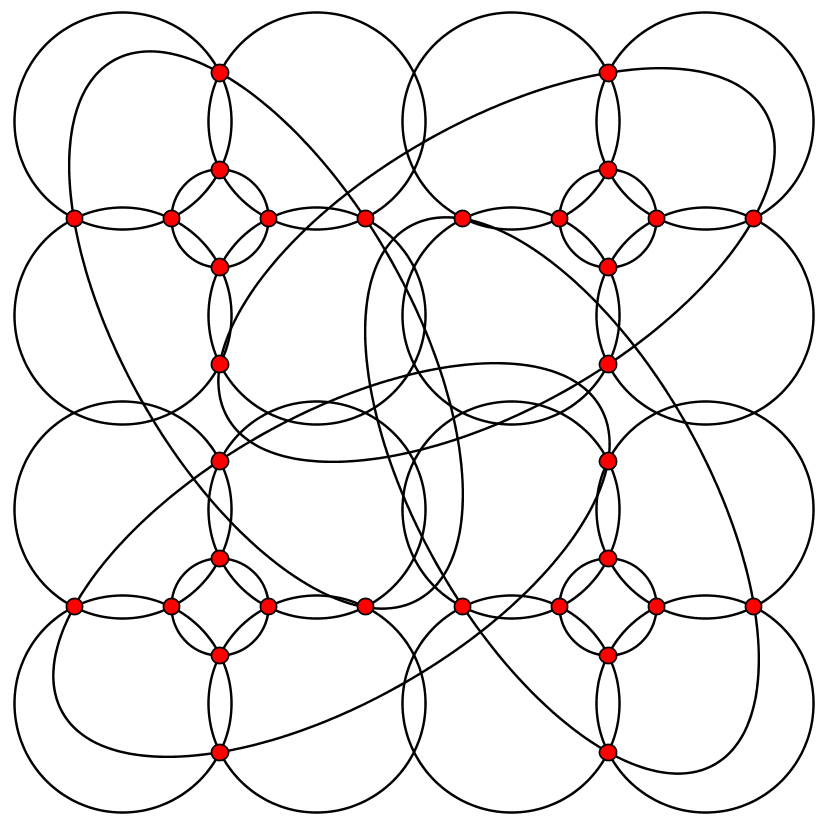}
\caption{A circular $(32_3,24_4)$ configuration that is obtained from four small Miquel configurations as a cyclic incidence cascade. 
            We call it the large Anti-Miquel configuration.}
\label{fig:AntiMiquelLarge}
\end{figure}

Note that the small Anti-Miquel configuration is only 2-connected. This means that its Levi graph disconnects if we delete two of its vertices.
The incidence switch that leads to the small Anti-Miquel configuration can easily be generalized to the \emph{cyclic incidence cascade}. 
Figure~\ref{fig:AntiMiquelLarge} depicts the large Anti-Miquel configuration by performing such a cyclic incidence cascade on four copies 
of the Miquel configuration. Obvously, the large Anti-Miquel configuration is also 2-connected. Since neither the small Anti-Miquel configuration 
nor the large Anti-Miquel configuration contains $K_{2,3}$ or $K_{3,2}$, their duals are also circular, too; hence they are all strongly circular. 

The following problem may be hard to solve (recall that a point-circle configuration is called \emph{isometric} if all its circles are 
of the same size \cite{gevay2014}.

\begin{problem}
Which strongly circular combinatorial configurations admit isometric realization by points and circles?
\end{problem} 

So far we have not found a 3-connected example.
Note that the dual Anti-Miquel configuration is a  strongly circular configuration that has no 
isometric geometric point-circle realization but admits geometric point-circle realization that is not isometric.

\begin{theorem}
Every lineal $3$-configuration has a geometric realization as a point-circle configuration.
\end{theorem}

\begin{proof}
Place the points of the configuration in the plane in general position: no three on a line, no four on a circle. The configuration circles are
determined by the triples of its configuration points. They are all distinct.
\end{proof}

\section{Point-Ellipse and Point-Conic Configurations}

When moving from point-circle to point-ellipse configurations we are faced with two new aspects. 
Namely two ellipses may intersect in up to four points. Thus ellipses may be distinguished from 
circles already in a combinatorial description. On the other hand, ellipses are essentially conics 
and so far we have found no combinatorial argument for distinguishing the two concepts.

The following theorem is well known in projective geometry~\cite{coxeter1974}.
\begin{theorem} \label{FivePointsThm}
In the real projective plane, any five distinct points, such that no three of them are collinear, 
determine a unique conic passing through them.
\end{theorem}

\begin{corollary}
\label{cor:intersect5}
If two non-degenerate conics $\mathcal{C}_1$ and $\mathcal{C}_2$ intersect in five different points, 
then they coincide, i.e., $\mathcal{C}_1 = \mathcal{C}_2$.
\end{corollary}

Let $C$ be a set of five points in the plane in linear position (i.e.\ no three of them are collinear). 
The conic that passes through the points in $C$ is denoted by $\mathcal{C}(C)$ in the present paper.

\begin{definition}
A combinatorial configuration $\mathcal{C}$ is \emph{conical} if any two distinct conics meet in at most four points.
\end{definition}

\begin{proposition}
A combinatorial configuration $\mathcal{C}$ is conical if and only if the coloured Levi graph $L(\mathcal{C})$ contains no $K_{5,2}$ subgraph.
\end{proposition}

\begin{definition}
A conical configuration $\mathcal{C}$ is called \emph{strongly conical} if it its dual is also conical.
\end{definition}

\begin{proposition}
A graph $G$ with a given black-and-white colouring is a coloured Levi graph of a \emph{combinatorial} conical configuration if and only
if $G$ is bipartite and $K_{5, 2} \nsubseteq G$.
\end{proposition}

By analogy we define strongly conical configuration as a conical configuration whose dual is also conical. 

\begin{proposition}
A graph $G$ is a coloured Levi graph of a \emph{combinatorial strongly} conical configuration if and only
if $G$ is bipartite and $K_{5, 2} \nsubseteq G$ and $K_{2, 5} \nsubseteq G$.
\end{proposition}

As we mentioned above, the problems of determining which lineal and circular configurations are realizable seem to be
very difficult. For certain conical configurations almost everything goes.

\begin{proposition} \label{prop:PointConic}
Every (biregular) bipartite graph which does not contain $K_{5, 2}$ and whose white vertices have degree at most five 
is a Levi graph of some point-conic configuration.
\end{proposition}

The proof simply follows from the well-known fact that $m$ points can be placed in the plane in such a way that no 6 lie on a common conic.

\begin{corollary}
Every $4$-valent (biregular) bipartite graph is a coloured Levi graph of some point-conic configuration.
\end{corollary}

\begin{proof}
By Proposition~\ref{prop:PointConic}, this graph admits a point-conic realization. 
By interchanging the vertex colours, the same result follows for the dual configuration.
Hence the configuration is strongly conical.
\end{proof}

\section{Isometric point-ellipse configurations}.

Extending the definition of isometric-point-line configurations, one may define a point-ellipse or point-conic configuration 
to be \emph{isometric} if all ellipses, or respectively conics, are congruent. We may define a point-conic configuration to 
be \emph{strongly isometric} if each congruence is obtained by translations only.

We give a simple construction for an infinite series of isometric point-ellipse configuration.

\begin{example}
Take two intersecting unit line segments. Thicken each segment in the same way to produce ellipses intersecting in four points
giving rise to a $(4_2,2_4)$ isometric point-ellipse configuration.

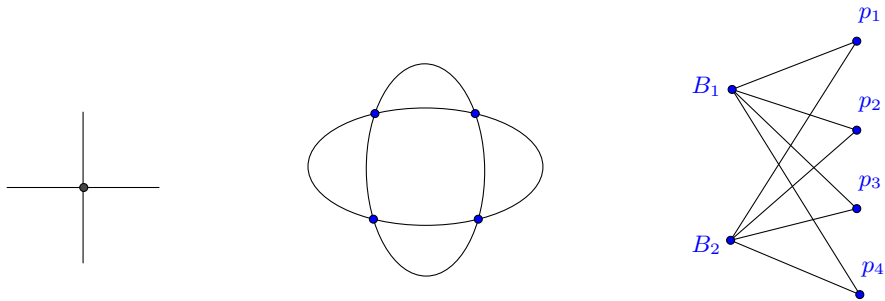
\begin{figure}[!htbp]
\centering
\definecolor{uuuuuu}{rgb}{0.26666666666666666,0.26666666666666666,0.26666666666666666}
\definecolor{qqqqff}{rgb}{0.0,0.0,1.0}
\begin{tikzpicture}[line cap=round,line join=round,>=triangle 45,x=1.0cm,y=1.0cm]
\clip(-3.24,0.2) rectangle (10.379999999999992,4.88);
\draw ({-1.6 - 1},2.36)-- ({-1.6 + 1},2.36);
\draw (-1.6,{2.36 - 1})-- (-1.6,{2.36 + 1});
\draw (6.94,3.66)-- (8.58,4.3);
\draw (6.94,3.66)-- (8.58,3.12);
\draw (6.94,3.66)-- (8.58,2.08);
\draw (6.94,3.66)-- (8.62,0.94);
\draw (6.92,1.66)-- (8.58,4.3);
\draw (6.92,1.66)-- (8.58,3.12);
\draw (8.58,2.08)-- (6.92,1.66);
\draw (8.62,0.94)-- (6.92,1.66);
\draw [rotate around={90.29896872358867:(14.571944781281413,2.635378436068815)}] (14.571944781281413,2.635378436068815) ellipse (1.1847581642989378cm and 0.785180526452594cm);
\draw [rotate around={179.40422307232993:(14.584335226410088,2.6781228957575722)}] (14.584335226410088,2.6781228957575722) ellipse (1.290980883492853cm and 0.7197088367000511cm);
\draw [rotate around={-0.1389781453350045:(2.9050261913503332,2.6363332109534587)}] (2.9050261913503332,2.6363332109534587) ellipse (1.5455854744905353cm and 0.7781798295424486cm);
\draw [rotate around={90.58994511638025:(2.9053372724229765,2.5927818607832562)}] (2.9053372724229765,2.5927818607832562) ellipse (1.4072956668433614cm and 0.7788386073131763cm);
\begin{scriptsize}
\draw [fill=qqqqff] (6.94,3.66) circle (1.5pt);
\draw[color=qqqqff] (6.6,3.7) node {$B_1$};
\draw [fill=qqqqff] (6.92,1.66) circle (1.5pt);
\draw[color=qqqqff] (6.6,1.6) node {$B_2$};
\draw [fill=qqqqff] (8.58,4.3) circle (1.5pt);
\draw[color=qqqqff] (8.759999999999993,4.64) node {$p_1$};
\draw [fill=qqqqff] (8.58,3.12) circle (1.5pt);
\draw[color=qqqqff] (8.759999999999993,3.4600000000000004) node {$p_2$};
\draw [fill=qqqqff] (8.58,2.08) circle (1.5pt);
\draw[color=qqqqff] (8.759999999999993,2.4200000000000004) node {$p_3$};
\draw [fill=qqqqff] (8.62,0.94) circle (1.5pt);
\draw[color=qqqqff] (8.799999999999994,1.28) node {$p_4$};
\draw [fill=uuuuuu] (-1.5911363636363636,2.36) circle (1.5pt);
\draw [fill=qqqqff] (13.92,3.3) circle (1.5pt);
\draw [fill=qqqqff] (15.22,3.3) circle (1.5pt);
\draw [fill=qqqqff] (15.26,2.06) circle (1.5pt);
\draw [fill=qqqqff] (13.88,2.08) circle (1.5pt);
\draw [fill=qqqqff] (2.24,3.34) circle (1.5pt);
\draw [fill=qqqqff] (3.56,3.34) circle (1.5pt);
\draw [fill=qqqqff] (2.22,1.94) circle (1.5pt);
\draw [fill=qqqqff] (3.6,1.94) circle (1.5pt);
\end{scriptsize}
\end{tikzpicture}
\caption{Left: $(1_2,2_1)$ configuration of a point and two isometric line segments. 
Center: Isometric point-ellipse $(4_2, 2_4)$ configuration. Right: The corresponding Levi graph. }  
\end{figure}
\end{example}

By using Example 1, we give a simple construction for an infinite series of isometric point-ellipse configurations.
The first member of this series is depicted in Figure~\ref{fig:ExampleTwo}. We start from a regular $n$-gon with 
sides elongated slightly and uniformly so as to obtain a configuration $(n_2)$. We slightly thicken the line segments 
into isometric ellipses so as to obtain an isometric point-ellipse configuration in which every two consecutive ellipses 
intersect precisely in four points.

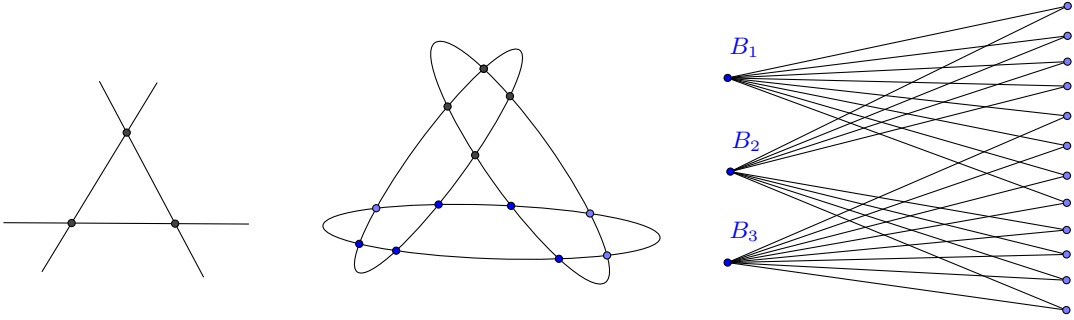
\begin{figure}[!htbp]
\label{fig:ExampleTwo}
\centering
\definecolor{uuuuuu}{rgb}{0.26666666666666666,0.26666666666666666,0.26666666666666666}
\definecolor{xdxdff}{rgb}{0.49019607843137253,0.49019607843137253,1.0}
\definecolor{qqqqff}{rgb}{0.0,0.0,1.0}
\begin{tikzpicture}[line cap=round,line join=round,>=triangle 45,x=1.0cm,y=1.0cm,scale=0.9]
\clip(0.29170000000000085,-1.2947236363636352) rectangle (16.476660000000006,3.8695563636363666);
\draw (11.04,2.46)-- (16.01128556850274,3.520037971379073);
\draw (16.009368393573343,3.0800463250824692)-- (11.04,2.46);
\draw (16.007712651588864,2.700053539644493)-- (11.04,2.46);
\draw (16.006143674205568,2.339973230177754)-- (11.04,2.46);
\draw (16.00422725870375,1.9001558725110947)-- (11.04,2.46);
\draw (16.002309324346772,1.4599899375845453)-- (11.04,2.46);
\draw (16.000391769703587,1.0199111469729694)-- (11.04,2.46);
\draw (15.998650402259297,0.6202673185086738)-- (11.04,2.46);
\draw (15.99176486223509,-0.9599641170467761)-- (11.04,-0.26);
\draw (15.993681657450697,-0.5200596150651444)-- (11.04,-0.26);
\draw (15.995337779148967,-0.13997968531219485)-- (11.04,-0.26);
\draw (15.996906376818472,0.22001347983957054)-- (11.04,-0.26);
\draw (15.998650402259297,0.6202673185086738)-- (11.04,-0.26);
\draw (16.000391769703587,1.0199111469729694)-- (11.04,-0.26);
\draw (16.002309324346772,1.4599899375845453)-- (11.04,-0.26);
\draw (16.00422725870375,1.9001558725110947)-- (11.04,-0.26);
\draw (11.08,1.08)-- (16.01128556850274,3.520037971379073);
\draw (11.08,1.08)-- (16.009368393573343,3.0800463250824692);
\draw (11.08,1.08)-- (16.007712651588864,2.700053539644493);
\draw (11.08,1.08)-- (16.006143674205568,2.339973230177754);
\draw (11.08,1.08)-- (15.996906376818472,0.22001347983957054);
\draw (15.995337779148967,-0.13997968531219485)-- (11.08,1.08);
\draw (15.993681657450697,-0.5200596150651444)-- (11.08,1.08);
\draw (15.99176486223509,-0.9599641170467761)-- (11.08,1.08);
\draw [rotate around={177.9310855946649:(7.587087662606778,0.1925009530459804)}] (7.587087662606778,0.1925009530459804) ellipse (2.465896968169812cm and 0.3953106776594227cm);
\draw [rotate around={-125.99174752138065:(6.811976569167591,1.2354387289606588)}] (6.811976569167591,1.2354387289606588) ellipse (2.0196769905877807cm and 0.39051579192640556cm);
\draw [rotate around={125.20039274088921:(8.005656236504642,1.2121208713249985)}] (8.005656236504642,1.2121208713249985) ellipse (2.161273388425279cm and 0.4853626475298307cm);
\draw (0.45686,0.32751272727273184)-- (4.03686,0.30751272727273227);
\draw (1.01686,-0.39248727272726613)-- (2.69686,2.3875127272727337);
\draw (1.85686,2.4075127272727332)-- (3.37686,-0.4724872727272662);
\begin{scriptsize}
\draw [fill=qqqqff] (11.04,2.46) circle (1.5pt);
\draw[color=qqqqff] (11.285760000000005,2.911236363636366) node {$B_1$};
\draw [fill=qqqqff] (11.08,1.08) circle (1.5pt);
\draw[color=qqqqff] (11.312380000000005,1.5269963636363657) node {$B_2$};
\draw [fill=qqqqff] (11.04,-0.26) circle (1.5pt);
\draw[color=qqqqff] (11.285760000000005,0.19599636363636536) node {$B_3$};
\draw [fill=xdxdff] (16.01128556850274,3.520037971379073) circle (1.5pt);
\draw [fill=xdxdff] (16.009368393573343,3.0800463250824692) circle (1.5pt);
\draw [fill=xdxdff] (16.007712651588864,2.700053539644493) circle (1.5pt);
\draw [fill=xdxdff] (16.006143674205568,2.339973230177754) circle (1.5pt);
\draw [fill=xdxdff] (16.00422725870375,1.9001558725110947) circle (1.5pt);
\draw [fill=xdxdff] (16.002309324346772,1.4599899375845453) circle (1.5pt);
\draw [fill=xdxdff] (16.000391769703587,1.0199111469729694) circle (1.5pt);
\draw [fill=xdxdff] (15.998650402259297,0.6202673185086738) circle (1.5pt);
\draw [fill=xdxdff] (15.996906376818472,0.22001347983957054) circle (1.5pt);
\draw [fill=xdxdff] (15.995337779148967,-0.13997968531219485) circle (1.5pt);
\draw [fill=xdxdff] (15.993681657450697,-0.5200596150651444) circle (1.5pt);
\draw [fill=xdxdff] (15.99176486223509,-0.9599641170467761) circle (1.5pt);
\draw [fill=qqqqff] (6.8136,0.5952963636363655) circle (1.5pt);
\draw [fill=qqqqff] (7.8736,0.5752963636363655) circle (1.5pt);
\draw [fill=qqqqff] (6.1936,-0.08470363636363454) circle (1.5pt);
\draw [fill=qqqqff] (8.5736,-0.20470363636363453) circle (1.5pt);
\draw [fill=qqqqff] (5.6536,0.01529636363636544) circle (1.5pt);
\draw [fill=xdxdff] (5.901947778700034,0.5404351453540278) circle (1.5pt);
\draw [fill=xdxdff] (9.028027064274013,0.4626805984389878) circle (1.5pt);
\draw [fill=xdxdff] (9.277100897355599,-0.1548696672118616) circle (1.5pt);
\draw [fill=qqqqff] (7.4736,2.5952963636363666) circle (1.5pt);
\draw [fill=uuuuuu] (7.4736,2.5952963636363666) circle (1.5pt);
\draw [fill=uuuuuu] (7.85686711560289,2.191088835558806) circle (1.5pt);
\draw [fill=uuuuuu] (7.348462934635833,1.3217478796497115) circle (1.5pt);
\draw [fill=uuuuuu] (6.945274571851051,2.038769798474687) circle (1.5pt);
\draw [fill=uuuuuu] (1.4486196635289397,0.3219721704932406) circle (1.5pt);
\draw [fill=uuuuuu] (2.9620242023346313,0.3135173965334325) circle (1.5pt);
\draw [fill=uuuuuu] (2.254099187996469,1.6548490026478435) circle (1.5pt);
\end{scriptsize}
\end{tikzpicture}
\caption{Left: $(3_2)$ configuration of three points and three isometric line segments. 
Center: Isometric point-ellipse $(12_2, 3_8)$ configuration. Right: The corresponding Levi graph. }
\end{figure}

This verifies the following proposition. 

\begin{proposition}
For any natural number $n\geq 3$ there exists an isometric point-ellipse $(4n_2,n_8)$ configuration.
\end{proposition}

Note that the isometric ellipses used in the construction above are not parallel. 
In the case of strongly isometric point-ellipse configurations any two ellipses 
meet in 0, 1 or 2 point.

\begin{theorem}
The classes of strongly isometric point-ellipse configurations and isometric point-circle configurations coincide.
\end{theorem}

\begin{proof}
By an appropriate dilation along an appropriate direction an ellipse turns into a circle. 
This transformation performed on the whole plane transforms all ellipses into circles 
and preserves the incidences. 
\end{proof}

\section{Intersection types}
\begin{definition}
Consider a non-disjoint pair of blocks $(B_1, B_2)$ in a configuration $\mathcal C$. We call the number of configuration 
points in which $B_1$ and $B_2$ meet the \emph{intersection type} of $(B_1, B_2)$. Let $\beta(\mathcal C)$ be the 
set consisting of those positive integers which occur as intersection type of two non-disjoint blocks of $\mathcal C$. We 
call this set the \emph{intersection type} of $\mathcal C$.
\end{definition}

\noindent
The following corollary is straightforward.
\begin{corollary}
There are 15 distinct intersection types to which a conical configuration can belong.
\end{corollary}

In what follow we construct an example whose intersection type is $\{1, 4\}$.

\begin{lemma} \label{EllipsesLem}
Choose a point on each side of a parallelogram $\mathcal P = ABCD$ such that they are in symmetric position
with respect to the centre of $\mathcal P$ and do not coincide with any of the vertices of $\mathcal P$. Then 
there are two uniquely determined ellipses passing through these points, moreover through $A,C$ and through
$B,D$, respectively.
\end{lemma}
\begin{proof}
It is sufficient to consider only one of the ellipses. By Theorem~\ref{FivePointsThm}, the quadruple of points 
chosen on the sides of $\mathcal P$, together with $A$, determines a unique ellipse $\mathcal E$. By condition, 
this quadruple is a centrally symmetric figure, and is in concentric position with respect to $\mathcal P$. Since 
any ellipse is also a centrally symmetric figure (and so is a parallelogram), it follows that $\mathcal E$ must pass 
through the vertex $C$, too (cf.\ Figure~\ref{parallelogram}).
\end{proof}
\begin{figure}[!h]
  \centering
    \includegraphics[width=0.4\textwidth]{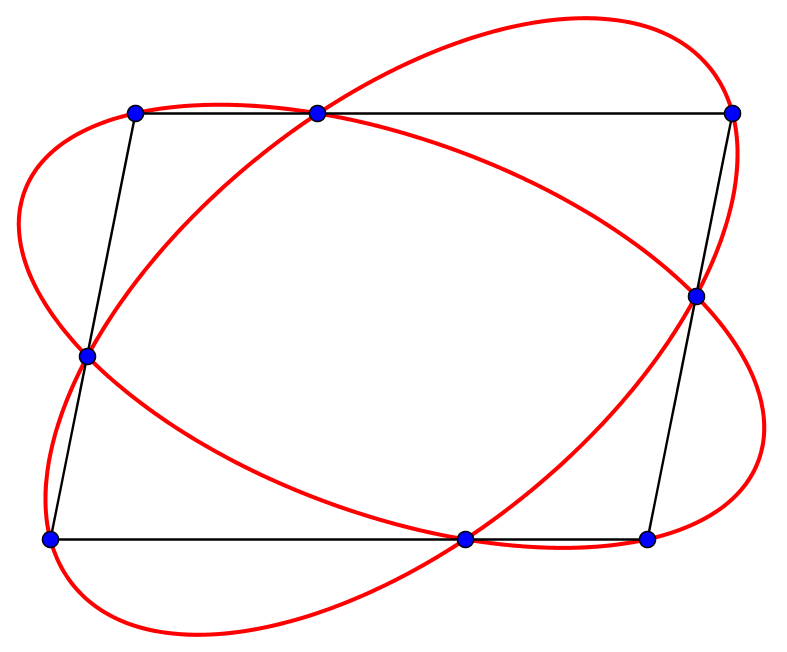}
  \caption{A pair of ellipses drawn on a parallelogram in concentric position.}
  \label{parallelogram}
\end{figure}

Start from the image of the 4-cube $Q_4$ in the plane under parallel projection. Denote this image by $Q'_4$.
The 24 square 2-faces of $Q_4$ are projected, in general, into parallelograms. Hence, taking the ellipses for all
these parallelograms as described in Lemma~\ref{EllipsesLem}, we obtain altogether 48 ellipses. Since $Q_4$
has 16 vertices and 32 edges, there are altogether 48 points which are incident with the ellipses. 

Moreover, we know that in $Q_4$, each edge is incident with precisely three 2-faces. Hence each of the 32 points
on the edges of $Q'_4$ is incident with 6 ellipses. On the other hand, each vertex of $Q_4$ is incident with 6 2-faces;
hence we have 16 additional points in $Q'_4$ such that each of them is incident with 6 ellipses, too. It follows that
we have a configuration of type $(48_6)$. 

It can directly be seen from the construction that intersecting pairs of ellipses belong to two types: pairs of one 
type have four configuration points in common (these are precisely those which belong to the same parallelogram), 
and pairs of the other type have a single configuration point in common (those which intersect in a vertex of $Q'_4$).
Hence the intersection type of this configuration is $\{1,4\}$.

\section{Carnot configurations} \label{CC}

We denote by $XY$ the \emph{signed distance} of points $X, Y$ of the Euclidean plane. This means that the 
line $\overline {XY}$ is supposed to be a directed line, and $XY = d(X, Y)$ or $XY = - d(X, Y)$ depending on
the direction of the vector $\overrightarrow {XY}$.

The following theorem is a generalization of the theorem of Menelaos in elementary geometry.
\begin{theorem}[Carnot's Theorem]
Let $ABC$ a triangle and let $A_1, A_2$ be points on the line $\overline {BC}$, $B_1, B_2$ on the line 
$\overline {CA}$ and $C_1$ and $C_2$ on the line $\overline {AB}$. The points $A_1, A_2, B_1, B_2, 
C_1$ and $C_2$ (all of which are supposed to be different from the vertices of the triangle) lie on the 
same conic if and only if
\begin{equation}
\frac{AC_1}{C_1B} \cdot \frac{AC_2}{C_2B} \cdot \frac{BA_1}{A_1C} \cdot \frac{BA_2}{A_2C} \cdot
\frac{CB_1}{B_1A} \cdot \frac{CB_2}{B_2A} = 1.
\end{equation}
\end{theorem}

\noindent
(For a proof of this theorem see e.g.~\cite{szilasi2012}.) 

By Richter-Gebert, Carnot's theorem can serve as the basis to build ``cycle theorems", the simplest of which
is as follows~\cite{richter2011}.
\begin{theorem}[Richter-Gebert] \label{JRG-Thm}
Given a tetrahedron, take two distinct points on each of the lines determined by its edges such that neither of these 
points coincide with a vertex. Suppose that for three of the faces of the tetrahedron, the 6-tuples of points belonging 
to these faces are coconical. Then the last, fourth 6-tuple of points is coconical, too.  
\end{theorem}
Figure~\ref{RichterThmFigure} visualizes this theorem. It can directly be seen that the theorem gives rise to a
point-conic configuration of type $(12_6, 4_3)$.

\begin{figure}[!h]
  \centering
    \includegraphics[width=0.45\textwidth]{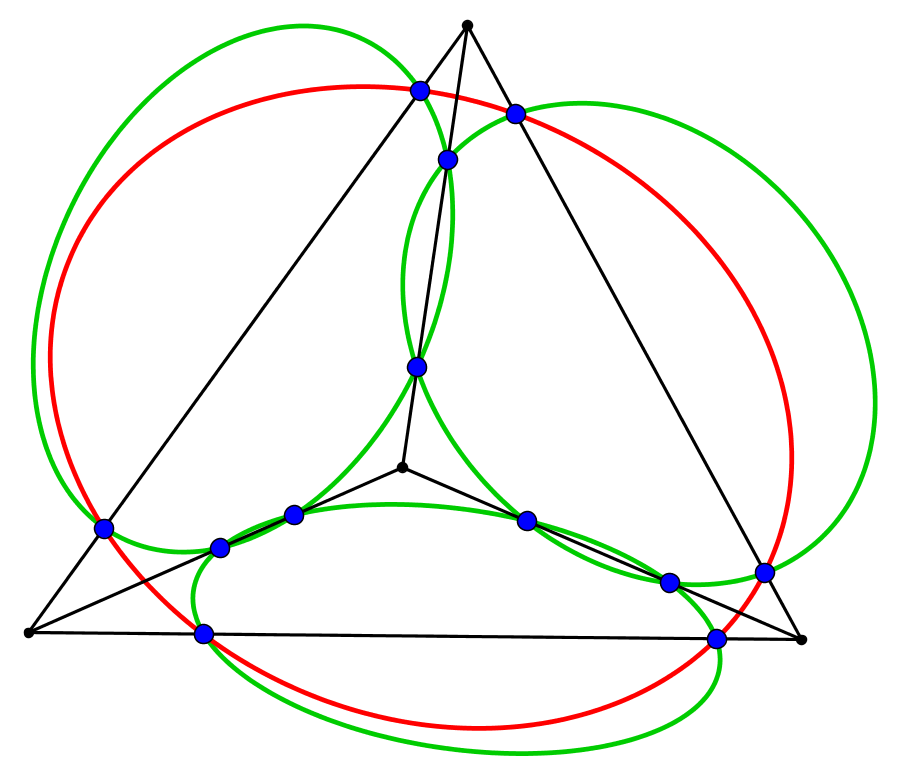}
  \caption{Visualization of the planar version of Theorem~\ref{JRG-Thm}}.
  \label{RichterThmFigure}
\end{figure}

\begin{remark} \label{TwoDimCarnot}
Theorem~\ref{JRG-Thm} can be extended to any oriented triangulated compact 2-manifold $\mathcal M$: 
if all but one of the 6-tuples of points belonging to the triangles of $\mathcal M$ are coconical, then the 
last one is also coconical~\cite{richter2011}.\end{remark}

We shall call a configuration of points and conics which can be constructed by repeated application of Carnot's theorem
a \emph{Carnot configuration}. 

Clearly, it follows from Remark~\ref{TwoDimCarnot} that one can construct infinitely many different Carnot configurations, 
moreover, infinite classes of them.

\begin{example} [\emph{Dipyramidal Carnot configurations}]
For each integer $n \ge 3$, there is a point-conic configuration of type $\bigl((6n)_2, (2n)_6\bigr)$ derived from an
$n$-gonal dipyramid. (Recall that a dipyramid is the symmetric difference of two piramids with equal and common basis; 
here we conceive these figures as 2-manifolds.) Indeed, an $n$-gonal dipyramid has $3n$ edges, thus we have $6n$ 
configuration points. On the other hand, there are $2n$ triangular faces, thus we have the same number of conics.
\end{example}

Using the Cartesian product construction introduced in~\cite{pisanski2013} and independently in~\cite{gevay2013}, 
the third power of the dipyramidal configurations provides an infinite series of balanced Carnot configurations.

\begin{theorem}
For each integer $n \ge 3$, there is a point-conic configuration of type $\bigl((216n^3)_6\bigr)$.
\end{theorem} 

Here we use the following variant of the Cartesian product construction, mentioned in neither of the 
references above, which provides a planar configuration if the original configurations are planar. Let 
$\mathcal C_1=(P_1, B_1, I_1)$ and $\mathcal C_2=(P_2, B_2, I_2)$ be two geometric configurations
given in the same plane. Then the point set $P$ of $\mathcal C_1 \times \mathcal C_2$ is the \emph
{Minkowski sum} of $P_1$ and $P_2$, i.e.\ it can be given as the following vector sum: 
$P = \{\mathbf {v}_1 + \mathbf {v}_2 \,|\, \mathbf v_1 \in P_1, \mathbf v_2 \in P_2\}$
(here a point $p$ is considered as given by its position vector $\mathbf v$). Let $\tau_{\mathbf v}$
denote a translation determined by a vector $\mathbf v$. With this notation, the set $B$ of blocks in 
$\mathcal C_1 \times \mathcal C_2$ is given in the following form:
$$
B = \{\tau_{\mathbf v_1}(B_2) \,|\, \mathbf v_1 \in P_1\} \cup \{\tau_{\mathbf v_2}(B_1) \,|\, \mathbf v_2 \in P_2\}.
$$

\section{An infinite class of point-conic 6-configurations} \label{6-cfg}

Let $m, n\geq 4$ be even integers, and let $\mathcal P_m$ and $\mathcal P_n$ be two regular $n$-gons 
with equal edge length. We conceive these polygons as convex 2-dimensional polytopes (in short, 2-polytopes). 
Take the \emph{Cartesian product} $\mathcal P_m \times \mathcal P_n$~\cite{coxeter1948, perles1967, 
ziegler1998}. This product is a convex 4-polytope which we shall denote by $\mathcal P_{m,n}$. It has 
$m+n$ facets which are right prisms such that they are semiregular 3-polytopes in the sense that their 
bases are regular $n$-gons and $m$-gons, respectively, and their side faces are squares. These prisms 
form two rings (topologically, two solid torii). The prisms within a ring are adjacent along their bases, 
while those belonging to different rings are adjacent along the square faces. By a special type of 
projection onto a hyperplane of $\mathbb E^4$, one obtains a \emph {Schlegel diagram}, which is a 
useful tool for visualizing a 4-polytope in $\mathbb E^3$~\cite{coxeter1948}. In Figure~\ref{schlegel} 
the Cartesian product of two hexagons is visualized in this way. Note that $\mathcal P_{4,4}$ is a special 
case, since it is equal to the 4-cube $Q_4$.
\begin{figure}[!h]
  \centering
    \includegraphics[width=0.5\textwidth]{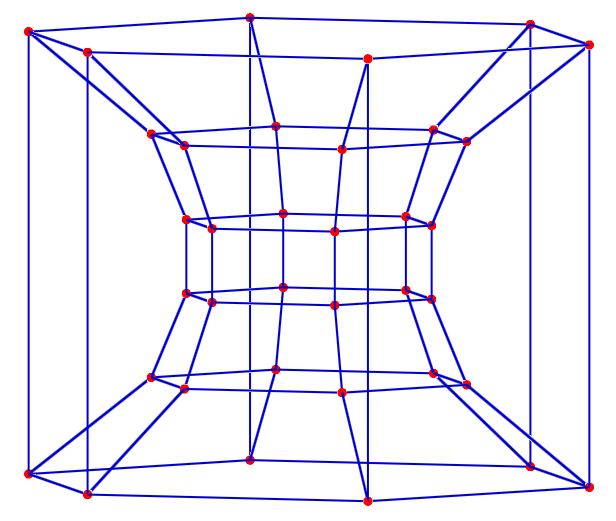}
    \caption{Schlegel diagram of the Cartesian product of two hexagons.}
    \label{schlegel}
\end{figure}

The next step of our construction is to inscribe hexagons in the prismatic facets of $\mathcal P_{m,n}$. 
These are axially symmetric hexagons such that in the generic case $m, n\geq 6$ there is a unique axis 
of symmetry passing through two opposite vertices. The hexagons are positioned so that these latter, 
distinguished vertices coincide with the midpoints of two opposite side edges of the prismatic facet 
(see Figure~\ref{HexagonPosition} for an example of such a hexagon in a hexagonal prism). The 
other four vertices of the hexagon coincide with the midpoints of the base edges adjacent to the 
former side edges. A consequence of this positioning of the hexagons is that they have another axis 
of symmetry perpendicular to the former. Note that in the special case $m, n=4$ these hexagons are 
regular (actually, they are inscribed in a \emph{Petrie polygon} of the facet of the 4-cube, which is 
a skew regular hexagon~\cite{coxeter1948}).

\begin{figure}[!h]
  \centering
    \includegraphics[width=0.45\textwidth]{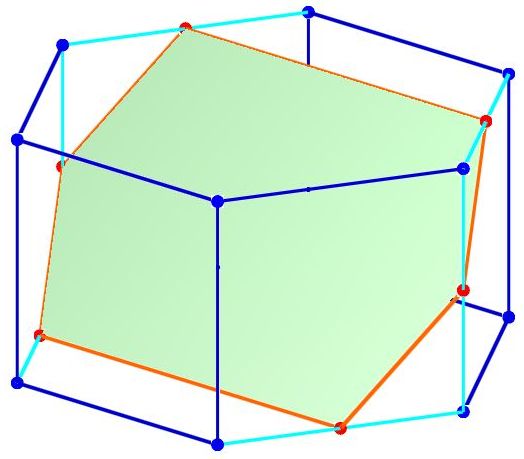}
    \caption{Hexagon in a prismatic facet of $\mathcal P_{m,n}$ with $m=6$.}
    \label{HexagonPosition}
\end{figure}

In this way, we inscribe $m$ hexagons in the $m$-sided prisms, and $n$ hexagons in the $n$-sided 
prisms, altogether $2mn$ hexagons in $\mathcal P_{m,n}$. Observe that $\mathcal P_{m,n}$ has 
$2mn$ edges (for, a base edge of a prism in the one ring of the facets coincides with a side edge of a 
prism in the other ring). Moreover, in each edge of $\mathcal P_{m,n}$ there are three prisms meeting. 
Hence, in the midpoint of each edge there are 6 inscribed hexagons meeting.

As a consequence of their symmetry, we can circumscribe an ellipse around each of the hexagons inscribed in 
$\mathcal P_{m,n}$. These ellipses, together with the midpoints of the edges of $\mathcal P_{m,n}$, form a
point-ellipse configuration of type $\bigl((2mn)_6\bigr)$ in $\mathbb E^4$. By a suitable projection onto a plane 
of $\mathbb E^4$, we obtain a planar point-conic configuration, which we shall denote by $\mathcal C_{2mn}$. 

\section{Examples of type $(32_6)$ and $(96_6)$}

As we noted in the previous Section, the inscribed hexagons are regular hexagons in case $m, n=4$; hence, 
in this case the circumscribed ellipses in $\mathbb E^4$ will be circles. Moreover, the projection onto the plane 
can be performed in such a way that all the circles project into ellipses. Using a particular projection of the 4-cube 
(see Figure~\ref{fig:CubeProjection}, the corresponding planar $(32_6)$ configuration is realized in the form as 
it is shown in Figure~\ref{pmncfg} (cf.~\cite{gevay2009}, Figures 14 and 15). 
\begin{figure}[!h]
  \centering
    \includegraphics[width=0.45\textwidth]{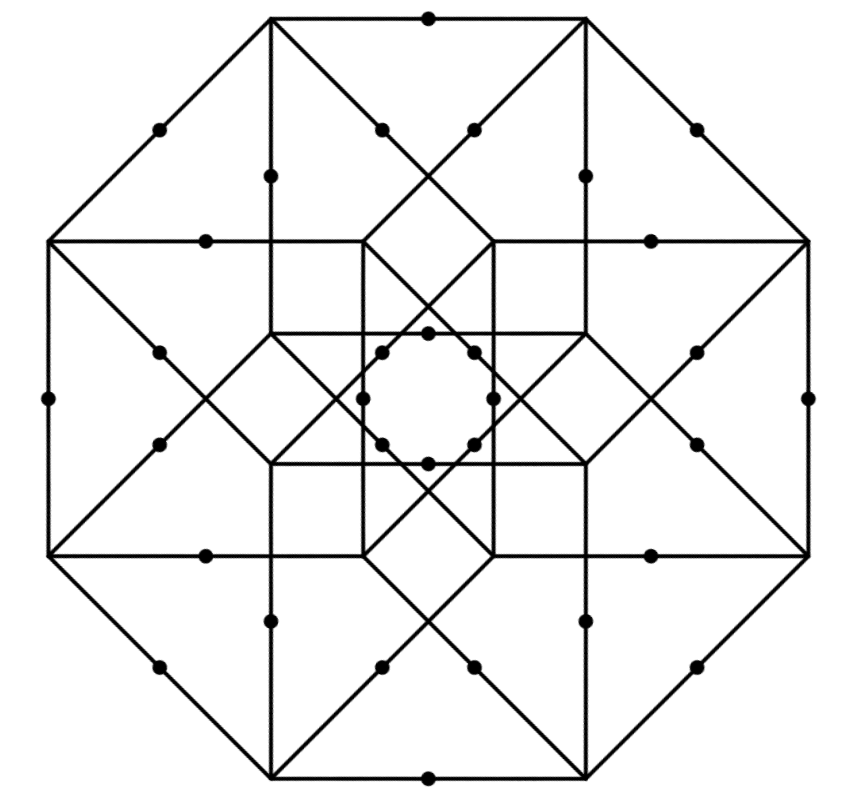}
    \caption{The projection of the 4-cube used in the construction of our $(32_6)$ example.
                The midpoints of the 32 edges of the 4-cube correspnd to the 32 configuration points.}
    \label{fig:CubeProjection}
\end{figure}
\begin{figure}[!h]
  \centering
    \includegraphics[width=0.85\textwidth]{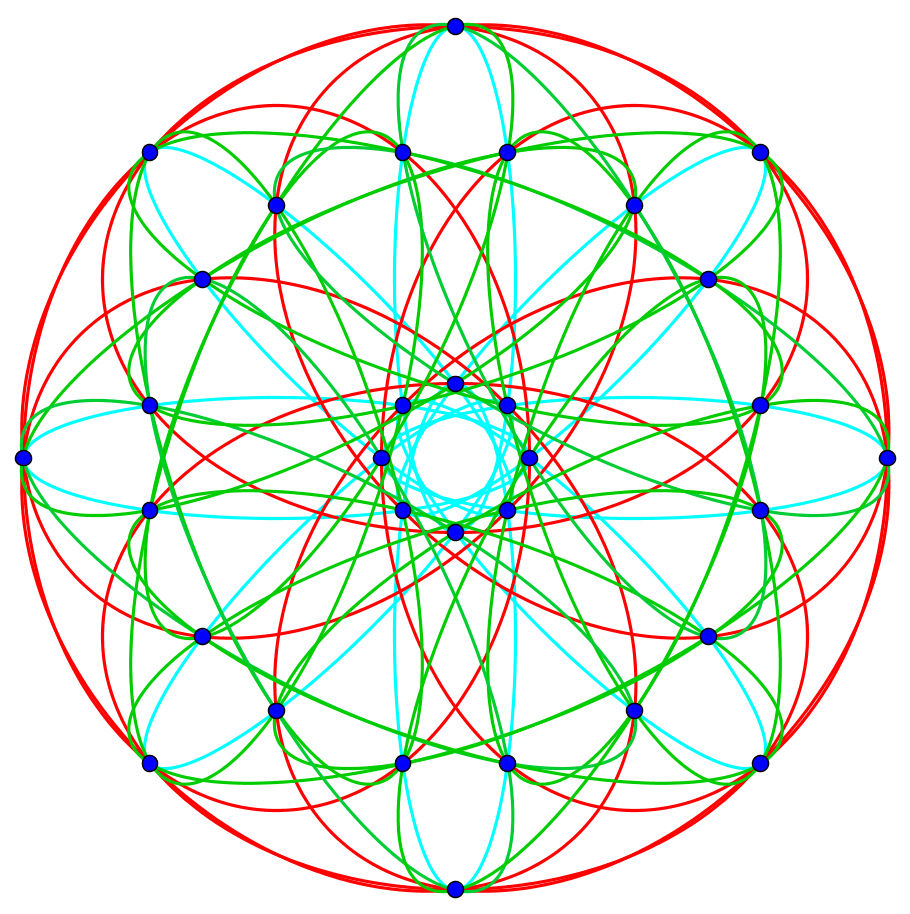}
    \caption{The $(32_6)$ point-ellipse configuration.}
    \label{pmncfg}
\end{figure}
The following observation reveals an interesting aspect of this configuration.
\begin{remark}
The system of the hexagons used in the construction of the $(32_6)$ configuration described above forms a geometric 
realization of a regular map of type $\{6,6\,|\,3,4\}$. This map was mentioned first by Coxeter, 1938~\cite{coxeter1938}, 
and it also occurs in~\cite{mcmullen1988}.
\end{remark}
From this observation the following property follows at once.
\begin{proposition}
The intersection type of this $(32_6)$ configuration is $\{1,2\}$.
\end{proposition}

This property may be useful in finding the answer to the following problem.

\begin{problem}
Can this configuration be realized as a point-circle configuration?
\end{problem}

The construction of $\mathcal P_{4,4}$ can be carried over from the 4-cube $Q_4$ to the case of the regular 24-cell as follows.
Recall that the regular 24-cell is a 4-polytope which has 24 regular octahedra as facets~\cite{coxeter1948}. Since it is self-dual, 
the number of its vertices is also 24. Moreover, it has 96 edges such that in each of them precisely three octahedra are meeting. 

Observe that a plane halving the distance between two opposite faces of a regular octahedron and parallel with them intersects
the octahedron in a regular hexagon. The same hexagon can also be inscribed in the octahedron in such a way that the octahedron 
is conceived as an antiprism, and the vertices of the hexagon coincide with the midpoints of the six side edges of this antiprism.
Since the octahedron has four pairs of parallel faces, there are precisely four regular hexagons inscribed in this way in a regular 
octahedron. Accordingly, one can inscribe in a regular 24-cell altogether 96 regular hexagons such that the vertices of these 
hexagons coincide with the midpoints of 96 edges of the 24-cell. Within each octahedral cell, there are precisely two hexagons
which are incident with the midpoint of a given edge. Since, each edge is incident with three octahedra, each midpoint is incident 
with precisely 6 hexagons.

It follows that circumscribing around each hexagon a circle, one obtains a point-circle configuration of type $(96_6)$ in 
$\mathbb E^4$. By a suitable projection, this configuration can be transformed into a planar point-conic configuration. 

We note that the intersection type of this $(96_6)$ configuration is also $\{1,2\}$. We did not consider the question of its 
realizability (in the plane) by ellipses; as for its realizability with circles, it may prove similarly difficult as that of the $(32_6)$ 
configuration above.

\bibliography{pointellipse}
\bibliographystyle{amsplain}

\end{document}